\documentclass{amsart}
\newcommand{\Q}{\mathbb Q}
\newcommand{\OO}{\mathfrak O}
\newcommand{\Z}{\mathbb Z}
\newcommand{\F}{\mathbb F}
\newcommand{\p}{\mathfrak p}
\newcommand{\m}{\mathfrak m}
\newcommand{\peu}{{\it peu ramifi\'ee}\ }
\DeclareMathOperator{\gal}{Gal}
\DeclareMathOperator{\gl}{GL}
\DeclareMathOperator{\pgl}{PGL}
\DeclareMathOperator{\tr}{Tr}
\newcommand{\HH}{\mathcal H}
\DeclareMathOperator{\frob}{Fr}
\newtheorem{theorem}{\sc Theorem}[section]

\newtheorem{definition}[theorem]{\sc Definition}

\theoremstyle{remark}

\newcommand{\be}{\begin{equation}}
\newcommand{\ee}{\end{equation}}

\begin{document}

\title{Finding Galois representations corresponding to certain Hecke eigenclasses \\
International Journal of Number Theory \\
Volume 5, Number 1, February 2009}
\author[M.~DeWitt]{Meghan DeWitt}
\address{Meghan De Witt\\St. Thomas Aquinas College \\ Sparkill, NY 10976\\ USA.}
\email{mdewitt@stac.edu}
\author[D.~Doud]{Darrin Doud}
\address{Darrin Doud\\ Brigham Young University\\ Provo, UT  84602\\ USA.}
\email{doud@math.byu.edu}
\date{Feb 2009}
\begin{abstract}
In 1992, Avner Ash and Mark McConnell presented computational
evidence of a connection between three-dimensional Galois
representations and certain arithmetic cohomology classes.  For some
examples they were unable to determine the attached representation.
For several Hecke eigenclasses (including one for which Ash and McConnell did not find the Galois representation), we find a Galois
representation which appears to be attached and show strong evidence for the uniqueness of this
representation. The techniques that we use to find defining polynomials for the Galois representations include a targeted Hunter search, class field theory, and elliptic curves.

\end{abstract}
\maketitle


\baselineskip 24pt
\section{Introduction}
 In \cite{AM}, Avner Ash and Mark McConnell provided computational evidence of a connection between three-dimensional Galois representations and certain arithmetic cohomology classes.  Their evidence was obtained by computing simultaneous eigenclasses of Hecke operators acting on certain arithmetic cohomology groups, and then attempting to compute Galois representations attached to these eigenclasses.  For many of the eigenclasses that they obtained, they were able to compute a Galois representation which seemed to be attached, but for some of their examples they were unable to determine such a representation.  In this paper, we examine one of these examples and find the Galois representation.  The Hecke eigenvalues arising from the cohomology computation allow us to predict the orders of Frobenius elements under the Galois representation, and we will find a Galois representation with the correct orders of Frobenius to correspond to the eigenclass for all primes $\ell$ less than 50. We will show that under certain assumptions on the Galois representations, the representations that we find are unique, and explain why these assumptions are reasonable.

\section{Attached eigenvectors}

 Let $\Gamma_0(N)$ be the set of matrices in $\gl_3(\Z)$ whose first row is congruent to $(*,0,0)$ modulo $N$, and let $\Sigma_N$ be the subsemigroup of integral matrices in $\gl_3(\Q)$ satisfying the same congruence condition.  For a fixed prime $p$, let $\HH(N)$ be the $\bar\F_p$-algebra of double cosets $\Gamma_0(N)\backslash \Sigma_N/\Gamma_0(N)$.  Then for any $\F_p[\Sigma_N]$-module $V$ and any $q\geq 0$, $H^q(\Gamma_0(N),V)$ is an $\HH(N)$-module.  We call the elements of $\HH(N)$ Hecke operators, and single out the elements corresponding to double cosets of the form $\Gamma_0(N)D(\ell,k)\Gamma_0(N)$, where $\ell | N$ is prime, $0\leq k\leq 3$, and $D(\ell,k)$ is a $3\times3$ diagonal matrix with the first $3-k$ entries equal to 1 and the remaining $k$ entries equal to $\ell$.  We write $T(\ell,k)$ for  the double coset corresponding to $D(\ell,k)$.
\begin{definition} Let $V$ be a $\HH(pN)$ module, and suppose that $v\in V$ is a simultaneous eigenvector of all $T(\ell,k)$ for all primes $\ell | pN$ and $0\leq k\leq 3$, so that $T(\ell,k)v=a(\ell,k)v$ for some $a(\ell,k)\in\bar\F_p$.  If $$\rho:G_\Q\to\gl_3(\bar\F_p)$$ is a Galois representation unramified outside $pN$, such that
$$\sum_{k=0}^3(-1)^k\ell^{k(k-1)/2}a(\ell,k)X^k=\det(I-\rho(\frob_\ell)X)$$
for all $\ell\neq pN$, then we say that $\rho$ is attached to $v$.
\end{definition}

Note that in this paper, the $\HH(pN)$-module $V$ will be taken to be $H^3(\Gamma_0(N),\F_p)$.  Also note that the definition of attached merely indicates a coincidence of the Hecke polynomial at $\ell$ with the characteristic polynomial of $\rho(\frob_\ell)$ for all $\ell$.  Given an eigenvector $v$, the problem of finding an attached $\rho$ (if one exists) can be very difficult.  In this paper, we will content ourselves with the following: given $v$, we will compute the Hecke polynomial.  Under the assumption that this is the characteristic polynomial of a $3\times 3$ matrix, we will predict the order of that matrix.  We will than find a Galois representation $\rho:G_\Q\to\gl_3(\bar\F_p)$ such that for each $\ell<50$, the order of $\rho(\frob_\ell)$ matches this prediction. Note that for any set of eigenvalues appearing with trivial coefficients, $a_i(\ell,0)=a_i(\ell,3)=1$.

\section{Systems of Hecke eigenvalues}

We begin by describing the Hecke eigenclasses in which we are interested.  In \cite{AM}, Ash and McConnell remark that there is a quasicuspidal eigenclass in $H^3(\Gamma_0(163),\F_5)$ and indicate that it has eigenvalues defined over $\F_5$.  Using software developed for \cite{ADP}, we have computed the space $H^3(\Gamma_0(163),\F_5)$, and found that it has six one-dimensional eigenspaces with eigenvalues defined over $\F_5$.  We list the eigenvalues for these eigenclasses, which we denote by $\alpha_1,\ldots,\alpha_6$ in Table~\ref{eigenvalues}.  Each of these sets of eigenvalues should have a three-dimensional Galois representation attached.

\begin{table}
\begin{tabular}{|c|c|c|c|c|c|c|c|c|c|c|c|c|c|c|c|}
\hline
$\ell$&2&3&5&7&11&13&17&19&23&29&31&37&41&43&47\cr
\hline
$a_1(\ell,1)$&4&4&$*$&1&0&3&4&0&0&2&0&1&4&1&0\cr
$a_1(\ell,2)$&1&1&$*$&0&0&3&1&2&4&0&0&0&4&2&3\cr
\hline
$a_2(\ell,1)$&1&1&$*$&0&0&3&1&2&4&0&0&0&4&2&3\cr
$a_2(\ell,2)$&4&4&$*$&1&0&3&4&0&0&2&0&1&4&1&0\cr
\hline
$a_3(\ell,1)$&2&2&$*$&1&2&3&4&4&3&4&2&0&2&1&0\cr
$a_3(\ell,2)$&2&0&$*$&0&2&3&1&3&3&3&2&3&2&2&3\cr
\hline
$a_4(\ell,1)$&2&0&$*$&0&2&3&1&3&3&3&2&3&2&2&3\cr
$a_4(\ell,2)$&2&2&$*$&1&2&3&4&4&3&4&2&0&2&1&0\cr
\hline
$a_5(\ell,1)$&3&2&$*$&2&0&3&2&0&2&2&1&0&2&3&0\cr
$a_5(\ell,2)$&4&0&$*$&2&0&3&2&2&0&0&1&3&2&3&3\cr
\hline
$a_6(\ell,1)$&4&0&$*$&2&0&3&2&2&0&0&1&3&2&3&3\cr
$a_6(\ell,2)$&3&2&$*$&2&0&3&2&0&2&2&1&0&2&3&0\cr
\hline
\end{tabular}
\caption{Six sets of eigenvalues in $H^3(\Gamma_0(163),\F_5)$.}\label{eigenvalues}
\end{table}

Ash and McConnell indicate that the representation potentially associated to their cohomology eigenclass seems to be reducible as a sum of a two-dimensional representation and a one-dimensional representation.  To see this, we let $\sigma:G_\Q\to \gl_2(\bar\F_5)$ be an ordinary Galois representation of Serre weight 2 \cite{Serre}, with determinant the cyclotomic character modulo 5 (denoted $\omega$).  According to the main conjecture of \cite{ADP}, there are two ways to obtain a three-dimensional Galois representation with trivial predicted weight from $\sigma$: we may take the direct sum $\theta=\omega^2\oplus\sigma$, or we may take the direct sum $\theta'=\omega\sigma\oplus 1$.  If we denote the trace of $\sigma(\frob_\ell)$ by $b(\ell)$, we see that $\det(I-\sigma(\frob_\ell)X)=1-b(\ell) X +\ell X^2$,
$$\det(I-\theta(\frob_\ell))=(1-b(\ell) X+\ell X^2)(1-\ell^2 X)=1-(b(\ell)+\ell^2)X+(\ell+\ell^2b(\ell))X^2-\ell^3 X^3$$
and
$$\det(I-\theta'(\frob_\ell))=(1-\ell b(\ell) X+\ell^3 X^2)(1-X)=1-(\ell b(\ell)+1)X+(\ell^3+\ell b(\ell))X^2-\ell^3 X^3.$$

If we can find a single set of $b(\ell)$, so that some $a_i(\ell,1)=b(\ell)+\ell^2$, $a_i(\ell,2)=1+\ell b(\ell)$, $a_j(\ell,1)=\ell b(\ell)+1$, and $a_j(\ell,2)=\ell^2+b(\ell)$, then we can begin to search for a two-dimensional Galois representation with $\tr(\sigma(\frob_\ell))=b(\ell)$.

In fact, there are three possible sets of traces indicated in Table~\ref{2d}.
\begin{table}
\begin{tabular}{|c|c|c|c|c|c|c|c|c|c|c|c|c|c|c|c|}
\hline
$\ell$&2&3&5&7&11&13&17&19&23&29&31&37&41&43&47\cr
\hline
$b_1(\ell)$& 0 & 0 & * & 2 & 4 & 4 & 0 & 4 & 1 & 1 & 4 & 2 & 3 & 2 & 1 \cr
$b_2(\ell)$& 3 & 3 & * & 2 & 1 & 4 & 0 & 3 & 4 & 3 & 1 & 1 & 1 & 2 & 1 \cr
$b_3(\ell)$& 4 & 3 & * & 3 & 4 & 4 & 3 & 4 & 3 & 1 & 0 & 1 & 1 & 4 & 1 \cr
\hline
\end{tabular}
\caption{Traces of Frobenius under conjectured two-dimensional representations.}\label{2d}
\end{table}If we denote by $\sigma_i$ the (conjectured) two-dimensional Galois representation associated to the $b_i(\ell)$, and by $\theta_i$, $\theta'_i$ the two direct sums derived from $\sigma_i$, we see that $\rho_{2i-1}$ and $\theta_i$ have the same characteristic polynomials, and that $\rho_{2i}$ and $\theta'_i$ have the same characteristic polynomials.  Therefore, we will concentrate on finding the two-dimensional representations $\sigma_1$, $\sigma_2$, and $\sigma_3$.

\section{Number Fields from Galois representations}
A Galois representation $\rho:G_\Q\to \gl_n(\bar\F_p)$ is a continuous homomorphism from $G_\Q$ (the absolute Galois group of $\Q$, or the Galois group of $\bar \Q/\Q$) to a matrix group over the algebraic closure of a finite field.  Continuity here is with respect to the standard profinite (or Krull) topology on $G_\Q$, and the discrete topology on $\gl_n(\F_p)$.  The continuity of the homomorphism implies that it has finite image.  This in turn implies that its kernel is an open normal subgroup of finite index.  Then the fixed field of the kernel is a finite Galois extension $K/\Q$, and we can factor $\rho$ as a composition $\rho=\eta\circ\pi$, where $\pi:G_\Q\to\gal(K/\Q)$ is the canonical projection, and $\eta:\gal(K/\Q)\to\gl_n(\F_p)$ is a representation of the finite group $\gal(K/\Q)$. We call $K$ the number field cut out by $\rho$.  Note that given the number field $K$ and a homomorphism $\eta:\gal(K/\Q)\to\gl_n(\bar\F_p)$, we immediately obtain a Galois representation $\rho$.  Note that when we speak of $\rho(\frob_\ell)$, with $\frob_\ell$ in $G_\Q$, we may actually work with $\eta(\frob_\ell)$ with $\frob_\ell\in\gal(K/\Q)$, since the canonical projection take Frobenius elements to Frobenius elements.  Note also that $\eta$ is injective, so that to find the order of $\rho(\frob_\ell)$, it suffices to find the order of a Frobenius at $\ell$ in $\gal(K/\Q)$.  Our goal in this paper is to find polynomials with splitting fields $K/\Q$, such that the order of Frobenius at each prime $\ell$ matches the predicted order of Frobenius from the Hecke eigenvalues described above. Given the Hecke eigenvalues, we know the trace and determinant of $\sigma_i(\frob_\ell)$, and from this can easily compute the characteristic polynomial and eigenvalues.  If $\sigma_i(\frob_\ell)$ has two distinct eigenvalues, we know that $\sigma_i(\frob_\ell)$ is diagonalizable, and its order is the least common multiple of the orders of the eigenvalues.  If $\sigma_i(\frob_\ell)$ has a repeated eigenvalue, then $\sigma_i(\frob_\ell)$ is either diagonalizable (in which case it has the same order as the eigenvalue) or has an upper triangular Jordan canonical form (in which case it has order 5 times the order of the eigenvalue).  Note that knowing only the characteristic polynomial does not allow us to distinguish these latter two possibilities.

Using this technique, we obtain the orders of Frobenius listed in Table~\ref{frob}, where $n_i(\ell)$ denotes the order of the image of a Frobenius at $\ell$ under $\sigma_i$.
\begin{table}
\begin{tabular}{|c|c|c|c|c|c|c|c|c|c|c|c|c|c|c|c|}
\hline
$\ell$&2&3&5&7&11&13&17&19&23&29&31&37&41&43&47\cr
\hline
$n_1(\ell)$&24&24&*&4&3&4&4&4,20&24&4,20&4&24&   6& 4&24\cr
$n_2(\ell)$& 8& 8&*&4&3&4&8&4,20& 4&4,20&3& 4&2,10&24&24\cr
$n_3(\ell)$& 4&24&*&4&6&4&8&12  & 4&  12&6&24&   6&24&24\cr
\hline
\end{tabular}\caption{Orders of $\sigma_i(\frob_\ell)$}\label{frob}
\end{table}
Note that the systems of eigenvalues $a_i(\ell)$ all take values in $\F_5$ (and continue to do so for arbitrarily large $\ell$, since they arise from a one-dimensional eigenspace).  This implies (by \cite[Lemma 6.13]{DS}) that the representations $\sigma_i$ (if they exist) may also be defined over $\F_5$.  Hence, the image of $\sigma_i$ is isomorphic to a subgroup of $\gl_2(\F_5)$.  Based on \cite[Conjecture 3.1]{ADP}, we deduce that the Serre level of $\sigma_i$ is 163 and the nebentype of $\sigma_i$ is trivial.  These facts easily imply that the image of inertia at 163 under $\sigma_i$ has order 5. Combining this with the fact that we predict that the order of some Frobenius under $\sigma_i$ is 24, we see immediately that the $\sigma_i$ are surjective onto $\gl_2(\F_5)$ (since Magma \cite{Magma} indicates that no proper subgroup of $\gl_2(\F_5)$ of order divisible by $5$ has an element of order 24).  Now there is an exact sequence
$$0\to \F_5^\times\to\gl_2(\F_5)\to S_5\to 0,$$
where the injection takes $a\in\F_5^\times$ to $aI$, and we use the well known fact that $\pgl_2(\F_5)\cong S_5$.  Hence the fixed field of $\sigma_i$ contains an $S_5$-extension of $\Q$, ramified only at $5$ and at 163 (with $e=5$).  We now make the additional assumptions that $\det(\sigma_i)=\omega$, and that $\sigma_i$ is ordinary, wildly ramified, and finite at 5.  Then the field cut out by $\sigma_i$ contains an $S_5$-extension in which $5$ has ramification index 20 (and is \peu in the sense of Serre), and $163$ has ramification index 5.  By \cite[Theorem 4.2]{wild} and \cite{Note}, we see that the $S_5$-extension must be the Galois closure of a quintic field with discriminant $5^5163^4$, in which both $5$ and $163$ have ramification index 5.

Under the assumption that $\sigma_i$ exists, and has the properties described above, we will show that there is a unique $\gl_2(\F_5)$-extension of $\Q$ with the correct orders of Frobenius to be cut out by $\sigma_i$.  In the conclusion, we will explain why these conditions on $\sigma_i$ are reasonable.

\section{Targeted Hunter searches}
A Hunter search is designed to find a polynomial defining a number
field $K$ of degree $n$ and discriminant $D$, if such a number field
exists.  This search is based on the use of Hunter's theorem and
relations between coefficients of polynomials that bound the
coefficients of the defining polynomial in question.  However, the
search space produced by this method is generally far too large to
examine.  Thus we refine the Hunter search by using the desired
ramification type for a certain prime $p$ in $K$.  By knowing the
desired ramification of $p$ in $K$, we could determine congruence
conditions of the defining polynomial of $K$.  This allowed us to
decrease the search space to such an extent as to make our search
feasible in a reasonable amount of time \cite{Doud-Moore}.

We also limited our search space using congruence relations based
upon the desired ramification of certain primes \cite{Doud-Moore}.
We wanted our quintic field extension to only ramify at 5 and 163, each with
$e=5$. By \cite[Theorem 2]{Doud-Moore}, we see that a defining polynomial for the field must be congruent to a polynomial $(x-a)^5$ modulo 5 and modulo 163.  This restriction on the polynomial places congruences on the coefficients, reducing the size of our search space greatly.  For more information on determining the coefficient bounds and the congruence conditions, see \cite{Doud-Moore}.

Using these techniques, we programmed a search for a quintic polynomial
using GP/Pari\cite{Pari}. We ran this search (which took about 36 hours) and obtained a list of
polynomials satisfying the desired condition.  We then tested the polynomials to make sure that they
had the desired discriminant of $5^5163^4$, eliminated polynomials defining isomorphic number fields,
and were left with three fields, defined by the following three quintic
polynomials:
\begin{align*}
g_1&=x^5+13040x^2-117360x+307744\cr
g_2&=x^5+10595x^2-104320x+254932\cr
g_3&=x^5+13040x^2-104320x-8319520.
\end{align*}
Each of these polynomials defines an  $S_5$-extension, and the order of Frobenius at $\ell$ in the splitting field of each $g_i$ divides the predicted order of Frobenius under $\sigma_i$ (with a quotient dividing 4, as expected).

We will find it useful to describe the splitting field of each of these quintic polynomials as the splitting field of a sextic polynomial.  To do this, we use a resolvent calculation, described in \cite[Algorithm 6.3.10(1)]{Cohen}.  This resolvent calculation yields a degree six polynomial with the same splitting field as the quintic.  When we perform this calculation on our quintic polynomials we obtain the following sextic polynomials:
\begin{align*}
f_1&=x^6-3x^5+5x^4+60x^3-95x^2-118x-91,\cr
f_2&=x^6-2x^5-25x^4+190x^3-315x^2-336x+1712,\cr
f_3&=x^6-3x^5-25x^4+5185x^3-62915x^2+585072x+213824.
\end{align*}

\section{Structure of $\gl_2(\F_5)$}
Using Magma \cite{Magma}, we note that $G=\gl_2(\F_5)$ has a single conjugacy class of subgroups of order $80$.  It is simple to see that a representative subgroup of this conjugacy class is of the form
$$H=\left\{\left(\begin{matrix}*&*\cr0&*\end{matrix}\right)\right\}.$$
We note that $H$ has a normal subgroup $J$ of order 20, where
$$J=\left\{\left(\begin{matrix}*&*\cr0&1\end{matrix}\right)\right\},$$
and we note that every subgroup of $G$ of order 20 is conjugate to this one.
Our conditions on $\sigma_i$ force the image of inertia at 5 under $\sigma_i$ to be a subgroup conjugate to $J$.  Hence, with a proper choice of basis, we may choose this image to equal $J$.  Let $K/\Q$ be the extension cut out by $\sigma_i$.  Then $K^J/\Q$ has degree 24 and is the inertia field of a prime lying over $5$.  Further, $K^J/K^H$ is Galois, with Galois group $H/J\cong\Z/(4\Z)$.  One checks easily that $J$ has no subgroups which are normal in $G$, so that the Galois closure of $K^J$ is equal to $K$.  Hence, we may describe $K$ by finding a degree 24 polynomial defining the inertia field of a prime above 5.  We already have a degree 6 polynomial defining a field contained in this inertia field, namely the $f_i$.  We will use class field theory to study the inertia field and to compute a degree 24 defining polynomial for it.

\section{Class field theory}

Denote $K^H$ by $F$ and $K^J$ by $L$.  Then $L/F$ is a cyclic extension of degree 4.  Since $5$ and $163$ are the only primes in $K/\Q$ that ramify, they must be the only primes that ramify in $L/F$.  However, for 163, $e=5$, so no primes over 163 can ramify in a degree four subextension.

We note that $F$ has two primes lying over $5$, one of ramification index 1 and one of ramification index 5.   We will write these primes as $\p_1$ and $\p_2$, with $\p_1$ having ramification index 1.  By our choice of $L$ as the inertia field of a prime lying over $5$, we see immediately that $\p_1$ cannot ramify in $L/F$.  Hence, the only finite prime of $F$ which ramifies in $L$ is $\p_2$.  Since the ramification is tame, we see then that $F$ lies inside the ray class field of $L$ modulo $\p_2\m_\infty$ (where $\m_\infty$ denotes the product of the infinite primes of $F$, so that we are allowing $\p_2$ and any infinite primes to ramify).  Note also that the narrow class number of $F$ (for $f_1$ and $f_2$) is 1, so that $L/F$ must be totally ramified.

Using GP/PARI to compute the ray class group of $F$ modulo $\p_2\m_\infty$, we find that for $f_1$ and $f_3$, it is cyclic of  order $4$.  Hence, $L/F$ is in fact the ray class field.  For $f_2$, the desired ray class group is $\Z/12\Z\times\Z/2\Z$, and the problem is more complicated, since this group has several quotients which are cyclic of order 4.  We use a different technique to find the defining polynomial for $\sigma_2$.  However, at this point we may examine the cyclic degree 4 extensions of $F$ unramified outside $\p_2$.  There are two such extensions (since $\Z/12\Z\times\Z/2\Z$ has two quotients isomorphic to $\Z/4\Z$), but only one of them yields the correct orders of Frobenius to correspond to the Hecke eigenvalues.  We have then proved the following theorem.
\begin{theorem} Assume that the $\sigma_i$ exist, have level 163 and determinant $\omega$, and are ordinary, wildly ramified and finite at 5.  Then there is at most one candidate $\gl_2(\F_5)$-extension for each $\sigma_i$.
\end{theorem}
\begin{proof} We have seen that such an extension must contain an $S_5$-extension with certain ramification at 5 and 163.  A Hunter search shows that there is exactly one such $S_5$-extension for each $\sigma_i$. Class field theory then shows that each $S_5$-extension is contained in at most one $\gl_2(\F_5)$-extension with the correct ramification and orders of Frobenius to correspond to $\sigma_i$.
\end{proof}

\section{Kummer Theory}

We require the following theorem:
\begin{theorem}
\cite[pg 114]{cox}\label{kummer}
Let $L=K(\sqrt{u})$ be a quadratic extension with $u\in
\OO_K$, and let $\mathfrak{p}$ be prime in $\OO_K$.

(i) If $2u \notin \mathfrak{p}$, then $\mathfrak{p}$ is unramified
in $L$.

(ii) If $2\in \mathfrak{p}$, $u\notin \mathfrak{p}$ and $u=b^2-4c$
for some $b,c\in \OO_K$, then $\mathfrak{p}$ is unramified
in $L$.  \end{theorem}

If we denote by $M$ the unique quadratic subextension of $L/F$, then we see that $M=L(\sqrt u)$ for some $u\in\OO_L$, and that $M/L$ is ramified only at $\p_2$.  Hence, $u$ must be an element of $\p_2$, but of no other prime ideal (by \ref{kummer}).  Since the class number of $F$ is 1, we see that $\p_2=(\alpha)$ is principal.  Hence, an element of $\OO_K$ which is contained in $\p_2$ but in no other prime ideal is of the form $\alpha^m\eta$, where $\eta$ is a unit in $\OO_K$. Adjusting this element by a square factor will not affect $m$, so we may take $u=\pm\alpha\eta$, where $\eta$ is a product of some subset of the fundamental units of $F$.  Note that there are only finitely many such elements.  For each such element, we compute a minimal polynomial for $u$, and substitute $x^2$ for $x$ to obtain a minimal polynomial for $\sqrt u$.  We then check to see if the resulting polynomial yields a field which is unramified at 2.  If it does, we check this polynomial to see if the orders of Frobenius in the resulting extension match those desired for $\sigma_i$.

For $f_1$ and $f_3$, we obtain unique polynomials from this process, each of which yields the correct orders of Frobenius. Hence, there is a unique quadratic extension of $L$ unramified outside $\p_2$.  This extension is the desired $M$.

Finally, we repeat this process to find a quadratic extension of $M$ ramified only at the unique prime of $M$ lying over $\p_2$.  Fortunately, each $M$ encountered has class number 1, so that the above procedure can be repeated.  We then obtain 2 degree 24 polynomials, one for each $\sigma_i$, as indicated in Table~\ref{deg24}.
\begin{table}
\begin{tabular}{|c|c|}
\hline
Representation&Defining Polynomial\cr
\hline
$\sigma_1$&$x^{24} - 637659x^{22} + 109056774377x^{20}$\cr&$ + 1632464535273540x^{18}$\cr
&$ - 371651092248574570x^{16} - 2672604891733833170x^{14} $\cr
&$- 5269788031324753370x^{12} + 155802031802967990x^{10}$\cr
&$ - 5228343306748595x^8 + 117734861534580x^6 $\cr
&$+ 817310749930x^4 - 3278260x^2 + 5$\cr
\hline
$\sigma_3$&$x^{24} - 368662181x^{22} - 56979878945733576x^{20}$\cr
&$-5926739223447260329773770x^{18}$\cr
&$+ 2825061262048524412523252201750x^{16}$\cr
&$  - 492299269650821506732949613908905505x^{14}$\cr
&$ - 1307345148590597879883355731944130x^{12}$\cr
&$ - 25233659029929802674589647281025x^{10}$\cr
&$ - 21690370211750470529946989210x^8 $\cr
&$- 60001189294636879166328970x^6  $\cr
&$+42796030038120191739040x^4 $\cr
&$+ 930599517955x^2 + 5$\cr
\hline\end{tabular}
\caption{Defining polynomials for $\sigma_1$ and $\sigma_3$.}\label{deg24}
\end{table}
Note that the class number of $F$ for $f_2$ is 6, so that we cannot find a generator for the ideal $\p_2$, and this process does not work.  We could use a more complicated technique of explicit class field theory, but we choose instead to use an entirely different method to find a defining polynomial.

\section{Elliptic Curves}
Let $E$ be an elliptic curve defined over $\Q$. Let $E(\bar\Q)_p=\{P_1,P_2,\ldots,
P_k\}$ be the $p$-torsion points on $E(\bar\Q)$. Then $E(\bar\Q)_p$ is an $\F_p$-vector space, on which $G_\Q$ acts as linear transformations.  This gives rise to a Galois representation $\tau:G_\Q\to\gl_2(\F_p)$.  If the conductor of $E$ is $N$, then $\tau$ is unramified outside $pN$.

Define
$P_i=(x_i,y_i)$ and
$$K=\mathbb{Q}(x_1,\ldots,x_k,y_1,\ldots,y_k)$$  Then $K$ is stable
under the action of $G_\mathbb{Q}$, so $K/\mathbb{Q}$ is a Galois
extension.  We note that $K$ is exactly the field cut out by $\tau$.  This implies that $\gal(K/\Q)$ is a Galois extension of $\Q$, with Galois group a subgroup of $\gl_2(\F_p)$.

We now examine the elliptic curve of conductor 163 given by the equation $$y^2+y=x^3-2x+1$$ (obtained from \cite{Cremona}).  The 5-torsion representation arising from this curve is ramified only at 5 and 163.  From \cite[Prop 2.11(c)]{DDT}, we see that in the fixed field $K$ of $\tau$, the prime $5$ has ramification index 20. From \cite[Theorem VII.3.4]{Silverman} we then see that for each torsion point, $py_i$ is an algebraic integer, and that the Galois conjugates of the $py_i$ are all of the form $py_j$. Using GP/PARI, we then compute the degree
24 polynomial
$$\prod_{y_i}(x-py_i)$$
to high enough precision to recognize the coefficients as integers, and round off to obtain the polynomial $f(x)$ in Table~\ref{ell24}, which has splitting field contained in the fixed field $K$ of $\tau$.
\begin{table}
\begin{center}
\begin{tabular}{|c|}
\hline
$x^{24} + 60x^{23} + 8475x^{22} + 402875x^{21} $\cr
$+ 13913355x^{20} + 354220875x^{19} + 8309320000x^{18}$\cr
$ + 169517221875x^{17} + 2491765593750x^{16}$\cr
$ + 24464219093750x^{15} + 179155477734375x^{14}$\cr
$ + 1413835025390625x^{13} + 14279768203125000x^{12}$\cr
$ + 132599856298828125x^{11} + 849863511181640625x^{10}$\cr
$ + 2785809996337890625x^9 - 3487057855224609375x^8 $\cr
$- 82465359191894531250x^7 - 407675512695312500000x^6 $\cr
$- 1095723202056884765625x^5 - 1700718978881835937500x^4$\cr
$ - 1327285671234130859375x^3 - 156774902343750000000x^2 +$\cr
$ 359581947326660156250x + 615432262420654296875$\cr
\hline
\end{tabular}\caption{Degree 24 polynomial defining $\sigma_2$}\label{ell24}
\end{center}
\end{table}
This polynomial is irreducible, and we easily compute that in the field defined by $f$, 19 has inertial degree 20 and 47 has inertial degree 24.  Hence, the Galois group of $f$ is a subgroup of $\gl_2(\F_5)$ containing elements of orders 20 and 24, so it must be all of $\gl_2(\F_5)$.  Thus, the splitting field of $f$ is exactly the field $K$ cut out by $\tau$.

By computing the splitting of primes in $K$ we were able to determine
that this polynomial defines a $\gl_2(\F_5)$-extension of $\Q$ which has the correct orders of Frobenius (for $\ell<50$) to correspond to be the field cut out by $\sigma_2$.

Note that we can explicitly compute the trace of $\tau(\frob_\ell)$ where $\tau$ is the $5$-torsion representation of an elliptic curve $E$.  This trace is the reduction mod $p$ of $a_\ell=p+1-|E(\F_\ell)|$.  When we do this, the numbers $a_\ell$ that we obtain are identical to the traces of $\sigma_2$, as desired.

\section{Conclusion}

Under the assumption that each $\sigma_i$ exists, has level 163 and determinant $\omega$, is ordinary, wildly ramified and finite at 5, we have shown that there is a single candidate for the field cut out by $\sigma_i$, and determined this field.  Note that these assumptions are not unreasonable.  The assumptions on the level and determinant are natural to make, based on \cite[Conjecture 3.1]{ADP}.

If we assume that $\sigma_i$ is ordinary and tamely ramified, then the ramification index at 5 would be $4$, and $\sigma_i$ would cut out a field containing an $S_5$-extension in which the ramification index at 5 is 4.  If we assume that $\sigma_i$ is supersingular (and hence tamely ramified), then $\sigma_i$ would cut out a field containing an $S_5$-extension in which the ramification index at 5 is 6.  In either case, the $S_5$-extension would be the Galois closure of a cubic field of discriminant $5^3163^4$.  A Hunter search shows that no such field exists.

Finally, if we assume that $\sigma_i$ is ordinary, wildly ramified at 5, and not finite, the Serre weight of $\sigma_i$ would be $5$ rather than $2$, and we would not expect $\omega^2\oplus\sigma_i$ to be attached to an eigenclass with trivial coefficients. The field cut out by $\sigma_i$ in this case would contain the Galois closure of a quintic field of discriminant $5^9163^4$.  A Hunter search for such a field could show that it does not exist, but would involve searching too many polynomials to be feasible.

\bibliographystyle{plain}
\bibliography{references}

\end{document}